\documentclass{article}
\usepackage{amsmath}
\usepackage{amssymb}
\usepackage{amsthm}
\usepackage{hyperref}
\usepackage{url}

\newtheorem{thm}{Theorem}
\newtheorem{lemma}[thm]{Lemma}
\newtheorem{prop}[thm]{Proposition}
\newtheorem{cor}[thm]{Corollary}

\theoremstyle{definition}
\newtheorem{dfn}{Definition}

\theoremstyle{remark}

\newcommand{\ZF}{\text{ZF}}
\newcommand{\ZFC}{\text{ZFC}}

\newcommand{\DC}{\text{DC}}
\newcommand{\HOD}{\text{HOD}}
\newcommand{\Reals}{\mathbb{R}}

\newcommand{\VitEq}{\mathrel{\mathbb{E}_0}} 
\newcommand{\EOne}{\mathrel{\mathbb{E}_1}}

\newcommand{\EFinPerm}{\mathrel{E_{S_{<\infty}}}} 
\newcommand{\dom}[1]{\text{dom}(#1)}
\newcommand{\nat}{\mathbb{N}}
\newcommand{\Coll}[2]{\text{Coll}({#1}, {#2})}
\newcommand{\card}[1]{\lvert{#1}\rvert}

\newcommand{\balanceEq}{\equiv_b}

\newcommand{\SE}{\mathrel{\text{SE}}}
\newcommand{\LessSE}{<_{\text{SE}}}
\newcommand{\LeqSE}{\le_{\text{SE}}}

\newcommand{\sortLex}{\le_{\text{lex}}^\uparrow} 
\newcommand{\revSortLex}{\ge_{\text{lex}}^\uparrow} 

\newcommand{\pbelow}{\precsim} 
\newcommand{\pabove}{\succsim}

\newcommand{\psbelow}{\prec} 
\newcommand{\pless}{\lesssim} 
\newcommand{\psless}{<} 

\title{On Strongly Equitable Social Welfare Orders Without the Axiom of Choice}
\author{Luke Serafin}
\date{\today}

\begin{document}
\maketitle

\begin{abstract}
Social welfare orders seek to combine the disparate preferences of an infinite sequence of generations into a single, societal preference order in some reasonably-equitable way.
In~\cite{strong-equity}, Dubey and Laguzzi study a type of social welfare order which they call SEA, for strongly equitable and (finitely) anonymous.
They prove that the existence of a SEA order implies the existence of a set of reals which does not have the Baire property, and observe that a nonprincipal ultrafilter over $\nat$ can be used to construct a SEA order.
Questions arising in their work include whether the existence of a SEA order implies the existence of either a set of real numbers which is not Lebesgue-measurable or of a nonprincipal ultrafilter over $\nat$.
We answer both these questions, the solution to the second using the techniques of geometric set theory as set out by Larson and Zapletal in~\cite{larson-zapletal-GST}.
The outcome is that the existence of a SEA order does imply the existence of a set of reals which is not Lebesgue-measurable, and does not imply the existence of a nonprincipal ultrafilter over $\nat$.
\end{abstract}


\section{Overview}

\subsection{Social Welfare Orders} \label{SWO-intro}

Let $\langle Y, \le \rangle$ be a totally-ordered set.
In theoretical economics we may think of $Y$ as a collection of utilities, where $x \le y$ means that the utility of $x$ is below that of $y$.
Different entities (perhaps individuals or generations) may derive varying utilities from the same societal choice.
For instance, consider a pair of policies $P_1$ and $P_2$ and a pair of entities $x_1$ and $x_2$.
Perhaps $P_1$ allows individuals to receive stock as compensation from a corporation without needing to pay tax until it is sold, while $P_2$ taxes stock grants at the market rate.
If $x_1$ is a wealthy individual and $x_2$ a government employee, it is likely that $x_1$ derives high utility from $P_1$ and low utility from $P_2$, while for $x_2$ the situation may be reversed due to the public projects which the
extra tax revenue from $P_2$ enables.
Or perhaps $x_1$ and $x_2$ are generations, with $x_2$ coming into existence after $x_1$, and $P_1$ is a lax policy on fishing rights in a certain region while $P_2$ is a stricter and more conservative policy.
It may be that $x_1$ derives considerable utility from the fishing revenue which can be obtained under $P_1$, but that $x_2$ derives low utility from $P_1$ due to depleted fish populations and would be much better off under $P_2$.

Since different individuals or generations benefit variably under a given choice of policy, how can conscientious policymakers decide which of a pair of policies is preferable for the entire population?
We shall make the idealization that the collection of individuals or generations is countably infinite, and represent it by $\nat$.
In this case the utilities derived from a given policy by all members of the population or by all generations can be represented as an element of $Y^\nat$.
A prelinear order on $Y^\nat$, interpreted as the preference order of results from different policies, is called a \emph{social welfare order}.
Note that unlike in the case of utilities, two distinct policies can be considered equally preferred by a social welfare order.
When $\pbelow$ denotes a prelinear order, we denote the corresponding relation of equal preference by $\approx$, which can be defined by ${\approx} = {\pbelow} \cap {\pabove}$.
The strict order notation $x \psbelow y$ means $x \pbelow y$ and $x \not\approx y$.
Of course, not all social welfare orders are equally desirable, and there are various properties which one might reasonably expect the preference order of a conscientious policymaker to possess.
The properties with which we shall be concerned are finite anonymity and strong equity.
The motivation for focusing on these principles is simply that Dubey and Laguzzi \cite{strong-equity} leave open two questions about social welfare orders which satisfy finite anonymity and strong equity.
For further details on social welfare orders and their properties, see \cite[sec. 6]{laguzzi}, \cite{strong-equity} and the references therein.

\begin{dfn}[{\cite{strong-equity}}]
A social welfare order is \emph{finitely anonymous} if and only if the labels given to individuals or generations don't affect the outcome, at least if we only change finitely many labels.
More precisely, an order $\pbelow$ is finitely anonymous if and only if for every finitely supported permutation $\pi$ of $\nat$ and for every $y \in Y^\nat$, $y \approx y \circ \pi$.
\end{dfn}

Finite anonymity has no relation to the utilities derived by individuals or generations, and is desirable simply so that certain distinguished individuals aren't given preference merely because they are distinguished.
The next property we consider does take utilities into account.
Intuitively, it favours equity by preferring scenarios where individuals have utilities which are closer together to those where they are farther apart, regardless of any quantitative differences between utilities.

\begin{dfn}[{\cite{daspremont-gevers-collective-choice}}]
The relation $\SE$ on $Y^\nat$ is defined by $x \SE y$ if and only if there are $i, j \in \nat$ such that $x \restriction (\nat \setminus \{i,j\}) = y \restriction (\nat \setminus \{i,j\})$ and $x(i) < y(i) < y(j) < x(j)$.
A social welfare order $\pbelow$ is \emph{strongly equitable} if and only if ${\SE} \subseteq {\psbelow}$, or equivalently if and only if ${\LessSE} \subseteq {\psbelow}$, where $\LessSE$ is the transitive closure of $\SE$.
\end{dfn}

\begin{dfn}[{\cite{strong-equity}}]
A social welfare order which is both strongly equitable and finitely anonymous is called a \emph{SEA order}.
\end{dfn}

\subsection*{Independence results}

While thinking about the possibility of actually using social welfare orders for policy decisions, at least in idealized scenarios, economists have noticed that many combinations of properties cannot be realized without
assuming a portion of the axiom of choice (e.g. \cite{zame-intergenerational-equity}), and in this sense are non-constructive.
In particular, the authors of~\cite{strong-equity} observe that the existence of a SEA order implies the existence of a set of reals which does not have the Baire property.
Since there are models\footnote{Assuming that $\ZF$ is consistent; we use $\ZFC$ in the metatheory, so if it is inconsistent everything is trivial.} of $\ZF + \DC$ in which every set of reals has the Baire property
(see, e.g. \cite{shelah-take-solovays-inaccessible-away}), $\ZF + \DC$ does not imply the existence of a SEA order.
The authors of~\cite{strong-equity} observe that the existence of a nonprincipal ultrafilter over $\nat$ is enough to guarantee the existence of a SEA order, and since it is well-known that the existence of such a nonprincipal
ultrafilter does not imply the full axiom of choice (see, e.g. \cite{barren-extension}), this shows that the existence of SEA orders is weaker than the axiom of choice.
But does the existence of SEA orders imply the existence of nonprincipal ultrafilters over $\nat$?
Dubey and Laguzzi in~\cite{strong-equity} leave this open, and we answer this question negatively, assuming the consistency of an inaccessible cardinal.
Another question left open in~\cite{strong-equity} is whether the existence of a SEA order implies the existence of a non-Lebesgue-measurable set of reals; we answer this question positively, with no large cardinal hypothesis needed.
Our work leaves the question of whether large cardinals are necessary open, though note that the existence of a single inaccessible cardinal is virtually the weakest large cardinal hypothesis which is considered.

\textit{Acknowledgement} This problem was suggested to me by Paul Larson, and is here solved using the techniques of geometric set theory developed by him and Jindřich Zapletal in \cite{larson-zapletal-GST}.
I am grateful to my advisor, Justin Moore, for a great deal of advice and encouragement without which these results would not have come to fruition.
This research was supported in part by NSF grants DMS–1854367 and DMS-2451350.

\section{Preliminaries}

The terminology and basic concepts which we use from set theory and descriptive set theory are mostly standard, and the reader is referred to the literature for more details, for example to~\cite{kunen} and~\cite{jech} for forcing
and ordinals, and to~\cite{kechris-classical} and~\cite{gao-invariant-descriptive} for Polish spaces and Borel reducibility of equivalence relations.
For large cardinals the standard source is~\cite{kanamori-higher-infinite}; we only need the notion of an inaccessible cardinal.
Our base theory for independence results is $\ZF + \DC$.
In the metatheory we assume full $\ZFC$.
For details on convergence with respect to filters and ultrafilters see~\cite[p. 51 ff]{engelking} and \cite{maciver-filter-convergence}.
The symmetric Solovay model $W$ is a model of $\ZF + \DC$ in which all sets of reals have standard regularity properties; in particular they have the Baire property and are Lebesgue measurable.
This model is discussed in~\cite{kanamori-higher-infinite}, \cite{jech}, and \cite{larson-zapletal-GST}.
Following Larson and Zapletal, we take the definition of $W$ to be $\HOD(\Reals \cup V)$ as evaluated in the Lévy collapse of an inaccessible cardinal.
We fix this inaccessible cardinal and denote it by $\kappa$ throughout the sequel.

Terminology and notation surrounding order relations isn't entirely standard, so we give some definitions.
\begin{dfn}
Let $Y$ be a set and ${\pbelow} \subseteq Y \times Y$. The pair $\langle Y, \pbelow \rangle$ is called a (binary) relational structure.
\begin{enumerate}
\item $\langle Y, \pbelow \rangle$ is a \emph{preorder} if and only if the relation $\pbelow$ is reflexive and transitive.
\item A preorder $\langle Y, \pbelow \rangle$ is a \emph{partial order} if and only if moreover $\pbelow$ is antisymmetric.
\item A partial order $\langle Y, \pbelow \rangle$ is \emph{linear} if and only if $\pbelow$ is total.
\item A preorder $\langle Y, \pbelow \rangle$ is \emph{prelinear} if and only if $\pbelow$ is total (so prelinear orders may fail to be antisymmetric).
\end{enumerate}
\end{dfn}

If $\langle Y, \pbelow \rangle$ is a preorder then the relation $\approx$ defined by $x \approx y$ if and only if $x \pbelow y$ and $y \pbelow x$ is an equivalence relation, and the preorder $\pbelow$ induces a partial order on
$Y / \approx$ which is linear if $\pbelow$ was prelinear. The word \emph{poset} may refer to either a preorder or a partial order; when antisymmetry is desired one may always take the quotient by $\approx$ in contexts where the word
``poset'' is used.

\begin{dfn}
An \emph{ordered Polish space} is a pair $\langle Y, \le \rangle$ with the following properties:
\begin{itemize}
\item $\langle Y , \le \rangle$ is a linear order (otherwise known as a total order).
\item $Y$ is a Polish space.
\item The Polish topology of $Y$ is the order topology induced by $\le$ (see e.g. \cite[5.15.f]{schechter-haf}).
\end{itemize}
\end{dfn}

In the sequel we shall assume without comment that all our ordered Polish spaces have at least two elements, and so $Y^\nat$ is infinite for every $Y$ under consideration.
The following proposition will allow us to reduce statements about general ordered Polish spaces to the important special case of the Cantor space $2^\nat$ with its lexicographic order.

\begin{prop} \label{polish-lo-cantor}
Let $\langle Y, \le \rangle$ be an ordered Polish space.
Then there is an order-preserving continuous embedding $f \mathrel{:} Y \rightarrow 2^\nat$ where $2^\nat$ has the lexicographic order.
\end{prop}

\begin{proof}
Let $S$ be the set of all elements of $Y$ which have immediate successors or predecessors.
Because $Y$ is separable, $S$ is countable.
Hence we may choose a countable dense subset $D$ of $Y$ containing $S$.
Identify $2^{<\nat}$ with elements of $2^\nat$ having finite support, and note that the lexicographic order on $2^{<\nat}$ is dense and has no endpoints.
By the universality of dense linear orders without endpoints for countable linear orders (e.g. \cite[thm. 2.5]{rosenstein-lo}),
we may fix an order-embedding $g \mathrel{:} D \rightarrow 2^{<\nat}$.
This lifts to an order-preserving map $\hat g = f \mathrel{:} \hat Y \rightarrow 2^\nat$ of Dedekind completions, defined by
\[ \hat g(y) = \sup \{ g(d) \mathrel{:} d \in D \cap (-\infty, y] \}, \]
which is Borel because it is order-continuous.
It remains to show that $\hat g$ is injective.
For $x, y \in Y$, if $g(x) = g(y)$ then for every $d \in D$, $d < x$ iff $d < y$.
Suppose for contradiction that $x \ne y$, say $x < y$.
Because $D$ is dense and there is no $d \in D$ with $x < d < y$, it must be that the interval $(x, y)$ is empty.
But then $x, y \in S \subseteq D$ and $g(x) = g(y)$ imediately implies that $x = y$.
\end{proof}

The space $[0,1]$ would serve as our universal ordered Polish space as well, since it is biembeddable with the Cantor space $2^\nat$.
But $2^\nat$ fits better with standard definitions in descriptive set theory, and is in some sense a simpler space than $[0,1]$, so we use it.
We denote the group of finitely-supported permutations of $\nat$ by $S_{<\infty}$, and by $\EFinPerm$ the orbit equivalence relation induced by the action of this group on $(2^\nat)^\nat$ by permuting coordinates.
Note that, taking the utility space $Y = 2^\nat$, finite anonymity is precisely the condition of $\EFinPerm$-invariance.

\subsection{SEA Orders are not Lebesgue-measurable}

In~\cite{strong-equity} Dubey and Laguzzi observe that the existence of a SEA order implies the existence of a set without the Baire property.
This implies that there is no SEA order in $W$, but what is more, the consistency of the non-existence of a SEA order with $\ZF + \DC$ follows because there is a model due to Shelah which can be constructed in $\ZF$ alone and in
which all sets of reals have the Baire property~\cite{shelah-take-solovays-inaccessible-away}.
Dubey and Laguzzi leave open whether the existence of a SEA order implies the existence of a set of reals which is not Lebesgue-measurable; we use Fubini's theorem to prove that it does, and observe that the same argument with
Fubini's theorem replaced by the Kuratowski-Ulam theorem gives a simple proof that a SEA order does not have the Baire property.

\begin{prop}
Let $\pbelow$ be a SEA order on $4^\nat$.
Then $\pbelow$ is not Lebesgue-measurable as a subset of $4^\nat \times 4^\nat$.
\end{prop}

\begin{lemma}
The relation $\EFinPerm \times \EFinPerm$ on $4^\nat \times 4^\nat$ is ergodic for Lebesgue measure.
\end{lemma}

\begin{proof}
Let $A \subseteq 4^\nat \times 4^\nat$ be Lebesgue-measurable and $(\EFinPerm \times \EFinPerm)$-invariant.
We must prove that $\lambda(A) \in \{0,1\}$, where $\lambda$ is the Lebesgue measure on $4^\nat \times 4^\nat$.
If $\lambda(A) = 0$ we are done, so assume that $\lambda(A) > 0$.
Fix $\varepsilon > 0$.
By the Lebesgue density theorem there are $K \in \nat$ and $s, t \in 4^{<\nat}$ with $|s| = |t| = K$ and such that
\[ \frac{\lambda(A \cap [s, t])}{\lambda{([s,t])}} > 1 - \varepsilon. \]
Here for $s', t' \in 4^{<\nat}$, $[s',t'] = \{ (x,y) \in 4^\nat \times 4^\nat : x \restriction |s'| = s' \wedge y \restriction |t'| = t' \}$.
The set
\[ B = \{(x, y) \in 4^\nat \times 4^\nat : \exists n \in \nat. \, \sigma^n(x) \restriction K = s \wedge \sigma^n(y) \restriction K = t \} \]
of pairs of points which eventually contain the patters $(s,t)$ (here $\sigma$ is the left shift) has measure $1$, and so we may replace $A$ by $A \cap B$.

For each $(x, y) \in B$ let $n_{x,y}$ be the least $n > K$ such that $\sigma^n(x) \restriction K = s$ and $\sigma^n(y) \restriction K = t$.
The map $(x, y) \mapsto n_{x,y}$ is measurable, and hence so is each set $C_n$ defined as the preimage of $n$ under the given map.
Note that the sets $\{C_n : n \in \nat\}$ are disjoint and $\bigcup_{n=K}^\infty C_n = 4^\nat \times 4^\nat$.
Moreover, for $s', t' \in 4^{<\nat}$ with $|s'| = |t'| = K$, $\lambda(C_n \cap [s',t']) = \lambda(C_n \cap [s,t])$ because we require $n_{x,y} > K$ for every $(x,y) \in B$.
Now for each $n > K$ let $\pi_n$ be the permutation of $\nat$ which interchanges the block $[0,K)$ with the block $[n,n+K)$.
Clearly the action of any permutation on $4^\nat \times 4^\nat$ by simultaneously permuting coordinates on both factors by the same permutation is measure-preserving.
Hence the map which acts on each $C_n$ by $\pi_n$ is also measure-preserving.
Therefore for any $s', t' \in 4^{<\nat}$ with $|s'| = |t'| = K$,
\begin{align*}
\lambda(A \cap [s',t']) &= \lambda\left( \bigcup_{n=K}^\infty (A \cap C_n \cap [s',t']) \right) & \\
&= \lambda\left( \bigcup_{n=K}^\infty \pi_n(A \cap C_n \cap [s', t']) \right) \\
&= \lambda\left( \bigcup_{n=K}^\infty (A \cap \pi_n(C_n) \cap [s, t]) \right) \\
&= \lambda(A \cap [s,t]),
\end{align*}
because acting simultaneously by $\pi_n$ on $C_n$ for each $n$ is measure-preserving, $A$ is invariant under finite permutations, $\bigcup_{n=K}^\infty \pi_n(C_n) = 4^\nat \times 4^\nat$,
and each $\pi_n$ acting on elements of $[s',t']$ replaces $s'$ by $s$ and $t'$ by $t$ in the first $K$ coordinates.
Consequently $\lambda(A) > 1 - \varepsilon$, and since $\varepsilon > 0$ was arbitrary, we conclude that $\lambda(A) = 1$.
\end{proof}

\begin{proof}[Proof of proposition]
Let $X = 4^\nat$ be the space of policies, and assume for contradiction that $\pbelow$ is Lebesgue-measurable as a subset of $X \times X$.
Because $\pbelow$ is finitely-anonymous, it is closed under $\EFinPerm$ in each coordinate.
By the lemma $\EFinPerm \times \EFinPerm$ is ergodic, so an $\EFinPerm \times \EFinPerm$-invariant set is either null or conull.
Note that $\EFinPerm \times \EFinPerm$ has Lebesgue measure zero as a subset of $4^\nat \times 4^\nat$ by Fubini's theorem, since each equivalence class and so each vertical and horizontal section is countable.

If $\pbelow$ is null, then by Fubini's theorem a conull collection of vertical sections are null.
Hence a conull collection of points in $X$ each have a conull collection of points above them.
This is only possible if a conull set of these points are pairwise equivalent, and so $\pbelow$ has a conull equivalence class.
But this is impossible, because there are disjoint open sets $U$, $V$ with each element in $U$ $\SE$-below an element of $V$.
For instance, following Laguzzi~\cite{strong-equity} we may take $U = [0,0,3,3]$, $V = [0,1,2,3]$.
The function $f \mathrel{:} U \rightarrow V$ defined by replacing $\langle 0, 0, 4, 4 \rangle$ by $\langle 0, 1, 2, 3 \rangle$ in the first four coordinates and leaving all other coordinates alone is a homeomorphism between
disjoint open sets which satisfies $x \SE f(x)$ for every $x \in U$.

Since $\pbelow$ is not null, it is conull, and in this case Fubini's theorem gives us that a conull collection of vertical sections are conull,
which means that a conull collection of points in $X$ have a conull collection of points below them.
Again this implies that $\pbelow$ has a conull equivalence class, which yields a contradiction as before.
\end{proof}

\subsection{Two Constructions of SEA Orders}

Dubey and Laguzzi in~\cite{strong-equity} observe that a nonprincipal ultrafilter over $\nat$ can be used to construct a SEA order.
Since they do not give a reference to an argument for this, we provide one here.

\begin{prop}
Assume there is a nonprincipal ultrafilter $U$ on $\nat$.
Then there is a SEA order on $(2^\nat)^\nat$.
\end{prop}

Since it will be used later, we make a definition before beginning the proof.

\begin{dfn} \label{dfn:sea-approx}
Let $Y$ be an ordered Polish space.
For $n \in \nat$, the order $\pless_n$ is defined by $x \pless_n y$ if and only if $x_{[n]} \le y_{[n]}$ in the lexicographic order of $Y^n$, where for $z \in Y^\nat$, $z_{[n]}$ is $z \restriction n$ written in sorted order,
meaning $(z \restriction n) \circ \pi$ for some permutation $\pi$ of $n$ such that for $i < j < n$, $((z \restriction n) \circ \pi)(i) \le ((z \restriction n) \circ \pi)(j)$.
\end{dfn}

It is straightforward to check that for each $n \in \nat$, $\pless_n$ is a prelinear order on $Y^\nat$.

\begin{proof}
Let $Y = 2^\nat$, so the space of policies is $Y^\nat$, and fix a nonprincipal ultrafilter $U$ on $\nat$.
The idea is to order points of $Y^\nat$ by the ultrafilter limit of their sorted initial segments.
Given $x, y \in Y^\nat$, define $x \pless y$ if and only if $\{ n \in \nat \mathrel{:} x \pless_n y \} \in U$, where finite tuples are compared lexicographically.
We shall show that $\pless$ is SEA.
It is a prelinear order because it is an ultralimit of the prelinear orders $\pless_n$.

To see that it is anonymous, suppose that $x \EFinPerm y$.
Then there is $N \in \nat$ such that for $n \ge N$, $x_{[n]} = y_{[n]}$.
Thus $x \pless_n y \pless_n x$ for $n \ge N$, and $x \pless y \pless x$ follows immediately from the fact that $U$ is nonprincipal and so contains all cofinite subsets of $\nat$.

It remains to verify strong equity.
Suppose that $x \SE y$, so $x$ and $y$ agree on all coordinates except $i$ and $j$, and that $x(i) < y(i) < y(j) < x(j)$.
Then for $n \ge i, j$, $x_{[n]}$ and $y_{[n]}$ agree up to the position where $x(i)$ appears in $x_{[n]}$, and at this position $y(i)$ appears in $y_{[n]}$.
Hence $x_{[n]} < y_{[n]}$, and since $n \ge i, j$ was arbitrary and $U$ is nonprincipal, $x \psless y$, as required.
\end{proof}

SEA orders can also be constructed from linear orders of $4^\nat / \VitEq$ or of $(2^\nat)^\nat / \EOne$, and in particular a transversal of $\VitEq$ is sufficient to construct a SEA order on $4^\nat$, the minimal space where the notion
of SEA order is nontrivial, and a transversal of $\EOne$ is sufficient to construct SEA orders for utilities drawn from any ordered Polish space.

\begin{prop} \label{lo-vitali-sea}
Suppose there is a linear order of $2^\nat / {\VitEq}$.
Then there is a SEA order on $4^\nat$ (and in fact on $n^\nat$ for any $n \ge 4$, including $n = \omega$).
\end{prop}

\begin{proof}
A slight modification of the construction of a SEA order from a nonprincipal ultrafilter over $\nat$.
Fix $n \in \nat \cup \{\nat\}$ with $4 \le n \le \omega$ and a linear ordering of $2^\nat / \VitEq$.
Order elements of $2^\nat$ first by their $\VitEq$-classes and then by $\pless_n$ for $n$ sufficiently large (this is well-defined because $x$ and $y$ being ordered by $\pless_n$ are $\VitEq$-equivalent).
The relation so-defined is obviously finitely-anonymous, and it is easy to see that it is strongly equitable.
That it is in fact a prelinear order follows from the fact that if $x \LessSE y$, then $x \VitEq y$.
\end{proof}

\begin{prop} \label{lo-eone-sea}
Suppose there is a linear order of $(2^\nat)^\nat / \EOne$.
Then for any ordered Polish space $Y$, there is a SEA order on $Y^\nat$.
\end{prop}

\begin{proof}
By proposition~\ref{polish-lo-cantor}
it suffices to show that there is a SEA order on $(2^\nat)^\nat$.
The argument proceeds in the same manner as that for the preceding proposition, only this time we
fix an ordering of $(2^\nat)^\nat / \EOne$ and order elements of $(2^\nat)^\nat$ first by their $\EOne$-classes and then by $\pless_n$ for sufficiently-large $n$.
\end{proof}

\section{Some Geometric Set Theory}

To make use of the machinery of geometric set theory, and in particular balanced forcing and its variants, it will be necessary to work
with forcing conditions in generic extensions.
In order for this to make sense, the forcing poset needs to be sufficiently-definable, and for our purposes that means it should be a
\emph{Suslin forcing}.

\begin{dfn}
A poset $\langle P, \le \rangle$ is \emph{Suslin} if and only if there is a Polish space $X$ over which $P$, $\le$, and $\perp$ are analytic.
\end{dfn}

The utility of the analyticity assumption is that Shoenfield absoluteness applies; more details and numerous examples can be found by following the references in~\cite{larson-zapletal-GST}.

Virtual conditions are, intuitively, objects which exist in $V$ and describe conditions of a Suslin forcing which are guaranteed to be
consistent across forcing extensions.
To formalize this, we start by defining objects in $V$, called $P$-pairs, which determine $P$-conditions in generic extensions of $V$.
Actually, we shall let $P$-pairs determine analytic sets of conditions in $P$, and for this we define an ordering on analytic subsets
of $P$ which is best thought of as ordering them by their suprema in a definable completion.

\begin{dfn}[{\cite[Definition 5.1.4]{larson-zapletal-GST}}]
For $A, B$ analytic subsets of a Suslin forcing $P$, the supremum of $A$ is below the supremum of $B$, denoted $\sum A \le \sum B$,
if and only if every condition below an element of $A$ can be strengthened to a condition below an element of $B$. In case $\sum A \le \sum B$
and $\sum B \le \sum A$, we write $\sum A = \sum B$.
\end{dfn}

\begin{dfn}[{\cite[def. 5.1.6]{larson-zapletal-GST}}]
A \emph{$P$-pair} for a Suslin forcing $P$ is a pair $\langle Q, \tau \rangle$ where $Q$ is a forcing poset and
$Q \Vdash \text{``$\tau$ is an analytic subset of $P$''}$.
\end{dfn}

The analytic set named in a $P$-pair is not guaranteed to have stable characteristics across generic extensions, an issue which the
notion of a $P$-pin seeks to resolve.

\begin{dfn}[{\cite[def. 5.1.6]{larson-zapletal-GST}}]
A $P$-pair $\langle Q, \tau \rangle$ for a Suslin forcing $P$ is a \emph{$P$-pin} if and only if
$Q \times Q \Vdash \sum \tau_\vartriangleleft = \sum \tau_\vartriangleright$, where
\[ \tau_\vartriangleleft = \{ \langle \sigma_\vartriangleleft, \langle p, q \rangle \rangle : \langle \sigma, p \rangle \in \tau, q \in P \} \]
is the lift of the name $\tau$ to the projection of a $Q \times Q$-generic filter to its left factor, and similarly for 
$\tau_\vartriangleright$.
As in~\cite{larson-zapletal-GST}, one may find it useful to think of these as the left and right copies of the name $\tau$.
\end{dfn}

\begin{dfn}
For $P$-pins $\langle P, \tau \rangle$, $\langle Q, \sigma \rangle$, define the relation of \emph{virtual equivalence} by
$\langle P, \tau \rangle \equiv \langle Q, \sigma \rangle$ if and only if $P \times Q \Vdash \sum \tau = \sum \sigma$.
Virtual conditions are equivalence classes of this relation.
\end{dfn}

That $\equiv$ is indeed an equivalence relation is established in \cite[Proposition 5.1.8]{larson-zapletal-GST}.
The intuition behind virtual conditions is that they describe (suprema of analytic sets of) conditions in a way that is independent
of the particular generic extension under consideration.
Note that for any poset $P$ and analytic subset $A \subseteq P$, the pair $\langle P, \check A \rangle$ determines a virtual
condition, so in particular $P$ embeds naturally into its set of virtual conditions (using the obvious observation that distinct
analytic subsets of $P$ determine distinct virtual conditions).

\begin{dfn}[{\cite[Definition 9.3.1]{larson-zapletal-GST}}]
Let $P$ be a Suslin forcing.
A virtual condition $\overline p$ of $P$ is \emph{placid} if and only if for all generic extensions $V[G]$, $V[H]$ such that $V[G] \cap V[H] = V$ and all conditions $p \in V[G]$, $q \in V[H]$, with $p, q \le \overline p$,
$p$ and $q$ are compatible.
$P$ is \emph{placid} if and only if for every condition $p \in P$ there is a placid virtual condition $\overline p \le p$.
The notions of \emph{balanced} virtual conditions and forcings are exactly analogous, with the requirement on the generic extensions $V[G]$, $V[H]$ strengthened to mutual genericity.
\end{dfn}

As we shall see in the course of the main proofs in this paper, balanced (and placid) pairs are of great utility in showing that specific
statements are forced, because if a statement is not decided by a balanced pair $\langle Q, \tau \rangle$ it is often possible to
use this fact to construct incompatible pairs below $\langle Q, \tau \rangle$.
In particular, a balanced pair for a forcing $P$ decides everything about the generic object for $P$ in the following sense:
\begin{thm}[{\cite[prop. 5.2.4]{larson-zapletal-GST}}]
Let $P$ be a Suslin poset and $\langle Q, \tau \rangle$ a balanced pair for $P$.
Then for any formula $\phi$ and parameter $z \in V$, one of the following holds:
\begin{itemize}
\item $Q \Vdash \Coll{\omega}{<\kappa} \Vdash \tau \Vdash_P W[\dot G] \models \phi(\dot G, \check z)$;
\item $Q \Vdash \Coll{\omega}{<\kappa} \Vdash \tau \Vdash_P W[\dot G] \models \phi(\dot G, \check z)$,
\end{itemize}
where $\dot G$ is the canonical $P$-name, in $V^Q$, for a $P$-generic filter.
\end{thm}

In order to demonstrate a forcing is balanced (or placid) it is often helpful to classify balanced pairs, and for that the following
equivalence relation is useful, as it provides a means of reducing balanced pairs to balanced virtual conditions.

\begin{dfn}[{\cite[Definition 5.2.5]{larson-zapletal-GST}}]
$P$-pairs $\langle Q, \tau \rangle$, $\langle R, \sigma \rangle$ are \emph{balance-equivalent}, denoted
$\langle Q, \tau \rangle \balanceEq \langle R, \sigma \rangle$, if and only if for all pairs
$\langle Q', \tau' \rangle \le \langle Q, \tau \rangle$, $\langle R', \sigma' \rangle \le \langle R, \sigma \rangle$,
\[ Q' \times R' \Vdash \exists q \in \tau' \ \exists r \in \sigma' \ \exists p. \ p \le q, r. \]
\end{dfn}

That $\balanceEq$ is indeed an equivalence relation is established in \cite[Proposition 5.2.6]{larson-zapletal-GST},
which also proves that if $\langle Q, \tau \rangle \le \langle R, \sigma \rangle$, then
$\langle Q, \tau \rangle \balanceEq \langle R, \sigma \rangle$.

An important property of balance equivalence is that every balance equivalence class includes a virtual condition, which is in fact
unique up to equivalence of virtual conditions, so when working with $P$-pairs up to balance equivalence it suffices to consider
virtual conditions.

\begin{prop}[{\cite[Theorem 5.2.8]{larson-zapletal-GST}}]
For any Suslin forcing $P$, every balance equivalence class of $P$-pairs includes a virtual condition which is unique up to
equivalence of virtual conditions.
\end{prop}

The main utility for us of the notion of placidity is that it entails that there are no nonprincipal ultrafilters over $\nat$.

\begin{prop}[{\cite[Theorem 12.2.8,3]{larson-zapletal-GST}}] 
If $P$ is a placid Suslin forcing and $G$ is $W$-generic over $P$, then in $W[G]$ there is no nonprincipal ultrafilter over $\nat$.
\end{prop}

\section{SEA No (Nonprincipal) Ultrafilter!}

We now turn to the problem of adding SEA orders to the symmetric Solovay model $W$ without adding nonprincipal ultrafilters over $\nat$ or $\VitEq$-transversals.
The most straightforward way to achieve this is to add a linear order of $\EOne$ (for a SEA order on $(2^\nat)^\nat$; a linear order of $\VitEq$ suffices for a SEA order on $n^\nat$ for countable $n$).
A direct, placid-forcing approach yields a more general result.
Let $W_{\VitEq}$ and $W_{\EOne}$ be the generic extensions of the symmetric Solovay model by the quotient space linearization posets of~\cite{larson-zapletal-GST}, example 8.7.5, for the equivalence relations displayed.
By corollary 9.2.12 in~\cite{larson-zapletal-GST}, $ \card{2^\nat / \VitEq} > 2^{\aleph_0}$ in both of these models, and by corollary 9.3.16 these are both placid extensions and so neither contains
a nonprincipal ultrafilter over $\nat$.
By lemma~\ref{lo-vitali-sea} the model $W_{\VitEq}$ contains a SEA order, and lemma~\ref{lo-eone-sea} together with lemma~\ref{polish-lo-cantor} demonstrate that the model $W_{\EOne}$ contains a SEA order on $Y^\nat$ for every
ordered Polish space $Y$.
This answers the question of Dubey and Laguzzi~\cite{strong-equity} about whether the existence of a SEA order implies the existence of a nonprincipal ultrafilter over $\nat$.

\section{More General Prelinearization}

In the previous section we showed how to answer the question of Dubey and Laguzzi using forcing machinery already developed by Larson and Zapletal in~\cite{larson-zapletal-GST}.
However, the author originally proceeded by constructing a new forcing specifically to add a SEA order, and this construction generalizes to a wider class of prelinearization problems so is worthwhile to write down.

\begin{dfn}
Let $\pbelow$ be a preorder on a set $X$.
A prelinear order $\pless$ on $X$ \emph{prelinearizes $\pbelow$} if and only if ${\pbelow} \subseteq {\pless}$ and ${\pless} \cap {\breve{\pless}} = {\pbelow} \cap {\breve{\pbelow}}$.
The order $\pless$ \emph{weakly prelinearizes $\pbelow$} if and only if ${\pbelow} \subseteq {\pless}$ and ${\psbelow} \subseteq {\psless}$, where $\psbelow$ and $\psless$ are the strict versions of $\pbelow$ and $\pless$, respectively.
\end{dfn}

Consider some specific sort of object, such as a nonprincipal ultrafilter over $\nat$, which is frequently useful for constructing (weak) prelinearizations.
The general form of the question we are interested in is when a prelinear order from $W$ has a (weak) prelinearization in some generic extension that contains no object of the specified sort.
Our result in this context is that all Borel preorders satisfying a technical condition can be prelinearized in a generic extension of $W$ containing no nonprincipal ultrafilter over $\nat$.

\begin{dfn}
A Borel preorder $\pbelow$ on a Polish space $X$ is \emph{tranquil} if and only if for every pair of generic extensions $V[G]$, $V[H]$ satisfying $V[G] \cap V[H] = V$,
and for every pair of elements $x \in V[G]$, $y \in V[H]$, if $V[K]$ is such that $V[G], V[H] \subseteq V[K]$, then (as evaluated in $V[K]$) if $x \pbelow y$ then there is $z \in V$ with $x \pbelow z$ and $z \pbelow y$.
\end{dfn}

Note that if a Borel preorder $\pbelow$ satisfies that $V$ is \emph{interval-dense} in any generic extension, in the sense that any nonempty interval contains an element of $V$, then $\pbelow$ is tranquil.
Unfortunately this is not true for $({\LeqSE} \cup {\EFinPerm})^*$, but by extending the relation while preserving its induced equivalence relation we obtain a tranquil order whose prelinearizations are SEA orders.

\begin{dfn}
For $x, y \in (2^\nat)^\nat$, $x \sortLex y$ if and only if there is $N \in \nat$ such that $x(n) = y(n)$ for $n \ge N$ and there are permutations $\pi, \pi' \in S_N$ such that $x \circ \pi$ and $x \circ \pi'$ are each increasing
sequences, and either $x = y$ or $(x \circ \pi)(k) < (y \circ \pi')(k)$ for $k$ least such that $x(k) \ne y(k)$.
\end{dfn}

What is going on here is that $x$ and $y$ must eventually agree to be compared with $\sortLex$, and then we sort initial segments large enough to contain all coordinates with differences and compare these sorted initial segments
lexicographically.
It is easy to see that if $x \SE y$ then $x \sortLex y$, so ${\LeqSE} \subseteq {\sortLex}$.
Moreover, if $x \sortLex y \sortLex x$ then there is a finite permutation $\pi$ of $\nat$ such that $y = x \circ \pi$, so clearly ${\sortLex} \cap {\revSortLex} = {\EFinPerm}$.

\begin{prop}
The preorder $\sortLex$ on $(2^\nat)^\nat$ is tranquil.
\end{prop}

\begin{proof}
It is immediate from the definition of $\sortLex$ that it is Borel.
Now suppose that we have generic extensions $V[G]$, $V[H]$ with $V[G] \cap V[H] = V$, that $V[G], V[H] \subseteq V[K]$, and that $x \in V[G]$, $y \in V[H]$, and $x \sortLex y$.
The case $x \EFinPerm y$ is trivial because all finite permutations are in $V$, so assume that $y$ is not a finite permutation of $x$.
Choose $N \in \nat$ such that $x(n) = y(n)$ for $n \ge N$, and note that the common value is always in $V$.
Fix permutations $\pi, \pi' \in S_N$ with the property that each of the sequences $(x \restriction N) \circ \pi$ and $(y \restriction N) \circ \pi')$ are sorted.
Because $y$ is not a finite permutation of $x$, there is a least number $k < N$ such that $x(k) < y(k)$.
Choose $z_k \in 2^{<\nat}$ with $x(k) < z_k < y(k)$ (really $z_k$ is extended by a tail of zeroes to an element of $2^\nat$).
Define $z \in (2^\nat)^\nat$ as follows:
\[ z(n) = \begin{cases} x(n) & \text{for $n < k$} \\ z_k & \text{for $k \le n < N$} \\ x(n) & \text{for $n \ge N$.} \end{cases} \]
By construction $z \in V$ and $x \sortLex z \sortLex y$.
\end{proof}

We shall demonstrate by constructing a placid forcing that any tranquil Borel preorder has a preliearization in a model which contains no nonprincipal ultrafilter over $\nat$ and no $\VitEq$-transversal.

\begin{dfn}
Let $\pbelow$ be a Borel preorder on a subset of a Polish space.
The \emph{prelinearizing poset} $P(\pbelow)$ is the poset of (enumerations of) preorders on countable subsets of $X$ which prelinearize the corresponding restriction of $\pbelow$, ordered by extension.
\end{dfn}

Note that this poset is $\sigma$-closed because a countable union of conditions is a condition.
Hence it preserves $\DC$. 

\begin{prop}
For $\pbelow$ an analytic preorder on a subset of a Polish space, $P(\pbelow)$ is Suslin.
\end{prop}

\begin{proof}
All the requirements to be an element of $P(\pbelow)$ are clearly analytic because $\pbelow$ is.
Extension is analytic also in this context, because conditions are required to be countable.
Two conditions $\le_p$ and $\le_q$ in $P(\pbelow)$ are incompatible precisely when there is a cycle in the relation ${<_p} \cup {<_q} \cup {\psbelow}$, and this is clearly an analytic requirement.
\end{proof}

\begin{prop}
For $\pbelow$ a tranquil Borel preorder on a Polish space $X$, $P(\pbelow)$ is placid, with placid virtual conditions classified by total prelinearizations of $\pbelow$.
\end{prop}

\begin{proof}
First we check that if $\pless$ is a prelinearization of $\pbelow$, then $\langle \Coll{\omega}{X}, \pless \rangle$ is a placid pair, where $X = \dom{\pbelow}$.
So suppose $V[G_0]$, $V[G_1]$ are separately generic extensions of $V$ with $V[G_0] \cap V[G_1] = V$, and that for $i < 2$, ${\pless_i} \in V[G_i]$ is a condition in $P$ with
${\pless_i} \le \overline p = \langle \Coll{\omega}{X}, \pless \rangle$ in the ordering of pairs for $i < 2$.
Strengthening if needed, we may assume that $\overline p$ is a condition in each $V[G_i]$.
If ${\pless_0} \perp {\pless_1}$ then there is a cycle $C$ in ${\psless_0} \cup {\psless_1} \cup {\pbelow}$.
If an element $x$ of this cycle occurs in the field of both $\psless_0$ and $\psless_1$, then because $V[G_0] \cap V[G_1] = V$, $x \in V$.
For links of the form $x \psbelow y$ in the cycle, by tranquility there is an element $z \in V$ with $x \psbelow z \psbelow y$, and hence since $\pless_i$ are prelinearizations,
$x \psless_i z \psless_{1-i} y$ for some $i < 2$.
Links of the form $x \psless_i y \psless_i z$ can be reduced to $x \psless_i z$.
Hence by a simple induction we may assume without loss of generality that the links of the cycle alternate between $\psless_0$ and $\psless_1$.
But because $x \psless_i y \psless_{1 - i} z$ implies $y \in V$, and the $\pless_i$ each extend $\pless$, this entails that there is a cycle in $\psless$, a contradiction.

If ${\pless_0} \ne {\pless_1}$ are distinct prelinearizations of $\pbelow$, then clearly no partial prelinearization can extend both simultaneously, so the balanced virtual conditions
$\langle \Coll{\omega}{X}, \pless_0 \rangle$ and $\langle \Coll{\omega}{X}, \pless_1 \rangle$ are not balance-equivalent.

Now suppose that $\overline p = \langle Q, \dot R \rangle $ is a placid virtual condition; we must find a condition of the form $\langle \Coll{\omega}{X}, \pless \rangle$, with $\pless$  a prelinearization of $\pbelow$, which is
balance-equivalent to $\overline p$.
Strengthening if necessary, we may assume that $X^V$ is countable after forcing with $Q$.
By placidity $\overline p$ decides the order of any pair of elements of $X^V$.
Let $\pless$ be the relation $\{ (x, y) \in X^V \times X^V \mathrel{:} \overline p \Vdash \check x \dot R \check y \}$.
Then $\overline p \le \langle \Coll{\omega}{X}, \pless \rangle$, so $\overline p$ is balance-equivalent to $\langle \Coll{\omega}{X}, \pless \rangle$.

Given $p \in P(\pbelow)$, by transfinite induction carried out in $V$ there is a prelinearization $\pless$ of $\pbelow$ extending $p$, and so $\langle \Coll{\omega}{X}, \pless \rangle \le p$.
Since we already saw that this pair is placid, we conclude that the poset $P(\pbelow)$ is placid.
\end{proof}

\begin{cor}
If $\pbelow$ is a tranquil Borel preorder then there is a model of $\ZF + \DC$ in which $\pbelow$ has a prelinearization but there is no nonprincipal ultrafilter over $\nat$.
\end{cor}

\begin{proof}
Because $P(\pbelow)$ is placid and $\sigma$-closed, forcing with it over $W$ yields the desired model.
\end{proof}

The natural approach to proving that forcing with $P(\pbelow)$ does not add an $\VitEq$-transversal is the notion of compact balance as developed in~\cite{larson-zapletal-GST}.
Unfortunately, the poset $P(\pbelow)$ does not appear to be compactly balanced in any obvious way, but fortunately Paul Larson pointed out to me that in fact closure of placid conditions under ultralimits is sufficient.
Since the following result is very general, we state it in terms of balance rather than placidity, and in fact use the notion of \emph{cofinal balance}, which means for a Suslin poset $P$ that for every generic extension of $V$
by a poset of cardinality less than $\kappa$, there is a further generic extension by a poset of cardinality less than $\kappa$ in which $P$ is balanced.

\begin{prop} \label{ulim-cl-no-tr}
Let $\langle P, \le \rangle$ be a cofinally balanced Suslin forcing below the inaccessible cardinal $\kappa$ with the following properties in a cofinal set of generic extensions $V[H]$ in which $P$ is balanced:
\begin{enumerate}
\item \label{b-gen-ext} If $V[H, H_1] \subseteq V[H, H_2]$ are generic extensions of $V[H]$ then for every balanced virtual condition $\overline p_0 \in V[H, H_0]$ there is a balanced virtual condition $\overline p_1 \le \overline p_0$
in $V[H, H_1]$,
\item \label{b-ulim-cl} The balanced virtual conditions in $V[H]$ are closed under limits with respect to ultrafilters in $V[H]$.
\end{enumerate}
Then $W^P \models \lvert 2^\nat / \VitEq \rvert > 2^{\aleph_0}$.
\end{prop}

\begin{proof}
By the hypotheses of the proposition we may assume that $P$ is balanced and that properties (\ref{b-gen-ext}) and (\ref{b-ulim-cl}) hold in $V$.
Suppose for contradiction that there is a $P$-name $\dot f$ and a condition $p \in P$ such that
\[ p \Vdash \text{``$\dot f$ is an injection from $2^\nat / \VitEq$ into $2^\nat$''}. \]
Choose $z \in 2^\nat$ such that $p, \dot f$ are definable from $z$, and let $K$ be a filter $V$-generic for a poset in $V$ of cardinality less than $\kappa$ and chosen such that $z \in V[K]$.

Now let $Q_R$ be the poset $\langle [\nat]^\nat, \subseteq \rangle$ and $Q_V$ be Vitali forcing, which consists of Borel $I$-positive subsets of $2^\nat$, ordered by inclusion, where $I$ is the $\sigma$-ideal over $2^\nat$
generated by Borel partial $\VitEq$-transversals.
See \cite[fact 9.2.3]{larson-zapletal-GST} and the references therein for details on this forcing.
For $\langle U, y \rangle$ $V[K]$-generic with respect to $Q_R \times Q_V$, note that $U$ is a nonprincipal ultrafilter over $\nat$ and $y \in 2^\nat$ in the generic extension.
Moreover, since $Q_R$ is $\sigma$-closed, $V[K]^{Q_R}$ has the same Borel codes as $V[K]$ and thus has the same notion of Vitali forcing $Q_V$.
Therefore $y$ is also $V[K][U]$-generic for $Q_V$.
Since Vitali forcing adds no independent reals\footnote{Subsets of $\nat$ which neither contain nor are disjoint from any infinite subset of $\nat$ in $V$} (see \cite{shelah-zapletal-ramsey}),
the set $U$ in $V[K][U][y]$ generates an ultrafilter over $\nat$.

Working in $V[K][U]$, let $\overline p_0$ be a balanced virtual condition below $p$.
Using (\ref{b-gen-ext}), choose a $Q_V$-name $\dot{{\overline p}_1}$ for a balanced virtual condition below $\overline p_0$ which is in $P$ as evaluated in $V[K][U][y]$.
For $n \in \nat$ define $y_n(i)$ to be zero if $i \le n$ and $y_n(i) = y(i)$ otherwise, so $y_n$ is obtained by zeroing out the first $n$ entries of $y$.
It is clear that this modification does not affect the genericity of $y$ (over $V[K]$ or $V[K][U]$), and that for every $n$ $V[K][U][y_n] = V[K][U][y]$.
Working now in $V[K][U][y]$, let $\overline p_2$ be the ultrafilter limit of $\langle \dot{{\overline p}_1} / y_n \mathrel{:} n \in \nat \rangle$.
This is a balanced virtual condition of $P$ by (\ref{b-ulim-cl}), and clearly $\overline p_2 \le \overline p_0, \overline p$.
It is immediate from the definition of $\dot{{\overline p}_1}$ that the same virtual condition would be obtained from any point of $2^\nat$ $\VitEq$-equivalent to $y$.
A contradiction is now reached exactly as in the proof of \cite[th. 9.2.2]{larson-zapletal-GST}. 
\end{proof}

\begin{lemma}
An ultralimit of prelinearizations of a preorder $\pbelow$ on a set $X$ is a preorder of $\pbelow$.
\end{lemma}

\begin{proof}
Let $\langle \pless_n \mathrel{:} n \in \nat \rangle$ be a sequence of prelinearizations of $\pbelow$, and fix a nonprincipal ultrafilter $U$ over $\nat$.
Take $\pless$ to be the ultralimit of the sequence, so for $x, y \in X$, $x \pless y$ if and only if $\{ n \in \nat \mathrel{:} x \pless_n y \} \in U$.
This is a prelinear order by the standard ultralimit argument, and clearly ${\pbelow} \subseteq {\pless}$ since this holds for each $\pless_n$.
It remains to check that if $x \pless y \pless x$, then $x \pbelow y \pbelow x$.
Suppose $x \pless y \pless x$, so
\[ \{ n \in \nat \mathrel{:} x \pless_n y \pless_n x \} \in U. \]
Since the relations $\pless_n$ are prelinearizations, if $x \pless_n y \pless_n x$ then $x \pbelow y \pbelow x$.
Hence
\[ \{ n \in \nat \mathrel{:} x \pbelow y \pbelow x \} \in U, \]
which means precisely that $x \pbelow y \pbelow x$.
\end{proof}

\begin{cor}
If $\pbelow$ is a tranquil Borel preorder then there is a model of $\ZF + \DC$ in which there is a prelinearization of $\pbelow$ but no $\VitEq$-transversal and no nonprincipal ultrafilter over $\nat$.
In particular, this holds for $\LeqSE$.
\end{cor}

\begin{proof}
The model obtained by forcing over $W$ with $P(\pbelow)$ will witness this.
We already saw that this model satisfies $\DC$ and contains no nonprincipal ultrafilter over $\nat$ in the proof of the last corollary.
Combining the classification of placid virtual conditions for $P(\pbelow)$ with proposition~\ref{ulim-cl-no-tr} and the last lemma yields that this model also contains no $\VitEq$-transversal.
\end{proof}

\section{Future Work}

Dubey and Laguzzi also define the notion of an ANIP social welfare order, which is finitely-anonymous and \emph{infinite Pareto}, meaning that if $x \le y$ coordinatewise and $x_i < y_i$ on infinitely-many coordinates $i$, then
$y$ is required to be strictly preferred to $x$ in the social welfare order.
They ask whether the existence of such an order implies the existence of a nonprincipal ultrafilter over $\nat$, and showing that it does not is beyond our current methods since the order defined by the infinite Pareto requirement
is not tranquil.
Hence the question about ANIP orders analogous to our main result about SEA orders remains open.

There is also the possibility of generalizing our results about prelinearizing tranquil Borel preorders in models with no nonprincipal ultrafilters over $\nat$ or $\VitEq$-transversals to broader classes of preorders, for example all
Borel preorders or all analytic preorders.
It would be intriguing if there is an obstruction, i.e. a Borel or at least analytic preorder such that $\DC$ suffices to prove that the existence of a prelinearization implies the existence of a nonprincipal ultrafilter over $\nat$,
or of a $\VitEq$-transversal.
The natural analogue of this question for weak prelinearization may also be of interest.

Another direction is to investigate SEA and related orders with utilities from an arbitrary definable (e.g. Borel or analytic) linear order on a Polish (or Suslin) space; what has been dropped is the requirement that
the space of utilities be an ordered Polish space.
In this context $2^\nat$ is no longer universal in the sense of proposition~\ref{polish-lo-cantor}, as witnessed by $[0,1] \times [0,1]$ with the lexicographic order, which is nonseparable
and therefore does not embed into the separable space $2^\nat$.

\bibliographystyle{plainurl}
\bibliography{math}

\end{document}